\documentclass[12pt]{article}

\usepackage{amsrefs}
\usepackage{amsmath,amssymb,amsthm}
\usepackage{enumerate,pxfonts}
\usepackage[all,cmtip]{xy}
\usepackage[applemac]{inputenc}
\usepackage[colorlinks=true, linkcolor=black, citecolor=black]{hyperref}

\usepackage{setspace}
\onehalfspace

\renewcommand \a{\mathfrak{a}}
\newcommand \Ad{\operatorname{Ad}}
\newcommand \adm{\mathrm{adm}}
\newcommand \al{\alpha}
\newcommand \bs{\backslash}
\def \be{\beta}
\newcommand \C{{\mathbb C}}

\newcommand \CB{\mathcal{B}}
\newcommand \CC{\mathcal{C}}

\newcommand \CE{\mathcal{E}}

\newcommand \CO{\mathcal{O}}

\def \diam{\operatorname{diam}}
\newcommand \eps{\varepsilon}

\def \g{\mathfrak g}
\newcommand \Ga{\Gamma}
\newcommand \GL{\operatorname{GL}}
\newcommand \ga{\gamma}
\newcommand \h{\mathfrak{h}}
\newcommand \Hom{\operatorname{Hom}}

\newcommand \ind{\operatorname{ind}}
\newcommand \la{\lambda}

\newcommand \m{\mathfrak{m}}
\newcommand \mqed{\tag*\qedhere}
\newcommand \n{\mathfrak{n}}
\newcommand \N{{\mathbb N}}
\newcommand \ol{\overline}
\newcommand \om{\omega}
\newcommand \Om{\Omega}
\def \op{\mathrm{op}}

\newcommand \R{{\mathbb R}}
\newcommand \sm{\smallsetminus}
\def \stab{\operatorname{stab}}

\def \tr{\operatorname{tr}}
\def \vol{\operatorname{vol}}
\newcommand \what{\widehat}
\newcommand \Z{{\mathbb Z}}

\renewcommand \Re{\operatorname{Re}}
\renewcommand \({\left(}
\renewcommand \){\right)}
\renewcommand \[{\big(}
\renewcommand \]{\big)}

\newcommand{\e}
[1]{\emph{#1}\index{#1}}

\newcommand{\norm}
[1]{\left\|#1\right\|}

\renewcommand{\sp}
[1]{\left\langle #1\right\rangle}

\newtheorem{theorem}{Theorem}[section]

\newtheorem{lemma}[theorem]{Lemma}

\newtheorem{proposition}[theorem]{Proposition}

\theoremstyle{definition}
\newtheorem{definition}[theorem]{Definition}

\begin{document}

\pagestyle{myheadings} \markright{GEOMETRIC ZETA FUNCTIONS WITH NON-UNITARY TWISTS}

\title{Geometric zeta functions with non-unitary twists}
\author{Anton Deitmar}
\date{}
\maketitle

{\bf Abstract:} It is shown that the zeta functions of Ruelle and Selberg admit analytic continuation to meromorphic functions on the plane for every compact locally-symmetric space and every non-unitary twist.


$$ $$

\tableofcontents

\newpage
\section*{Introduction}

In 
\cite{FP} it is shown that for a Fuchsian group the Selberg zeta function admits an analytic continuation for non-unitary twists.
If the group has cusps, an extra condition on the cusp-monodromy is needed.
In the papers 
\cites{Spilioti,Spilioti2} the continuation is given for odd dimensional compact hyperbolic manifolds.
In this note, we use a different technique, which works well for arbitrary compact locally symmetric spaces. Using some metric estimates due to Lubotzky, Mozes and Raghunathan, we are able to invoke the calculus of Lefschetz formulas for torus actions developed earlier.
The present approach will also work for groups of higher splitrank.

\section{Notation and basic facts}\label{sec1}
\begin{definition}
Let $\Ga$ be a group acting on a set $X$.
A point $x_0\in X$ is a \e{free point}, if 
 the stabilizer $\Ga_{x_0}=\stab_\Ga(x_0)$ is trivial.
\end{definition}

Recall that two metrics $d,d'$ on a space $X$ are called \e{Lipschitz equivalent}, if there is $C>0$ such that
$$
\frac1Cd\le d'\le Cd.
$$
Recall further that a metric $d$ on $X$ is called \e{proper}, if for every $x\in x$ and every $r>0$ the closed ball $\ol B_r(x)=\big\{ y\in X: d(x,y)\le r\big\}$ is compact.

\begin{definition}
\cite{Lub}.
A proper metric space $(X,d)$ is called a \e{path space}, if there exists $C>0$ such that for any $x,y\in X$ one has
$$
d(x,y)=\inf\left\{\sum_{j=0}^{n-1}d(x_{j},x_{j+1}):n\in\N, x=x_0, x_n=y, x_j\in X, d(x_j,x_{j+1})\le C\right\}.
$$
Any metric, which is Lipschitz equivalent to a path metric, is a path metric as well.
\end{definition}

Recall that a group $\Ga$ acts \e{properly discontinuously} on a space $X$, if every point $x\in X$ has a neighbourhood $U$ such that the set
$$
\big\{\ga\in\Ga: U\cap\ga U\ne\emptyset\big\}
$$
is finite.

\begin{proposition}[Proposition 3.2 of \cite{Lub}]
\label{prop1.3}
Let $\Ga$ be a finitely generated group acting properly discontinuously via isometries on a path space $(X,d_X)$ such that $\Ga\bs X$ is compact.
Let $S$ be a finite symmetric set of generators and let $d_S$ be the corresponding word metric on $\Ga$.
Let $x_0\in X$ be a $\Ga$-free point. Then the metric $d(\ga,\tau)=d(\ga x_0,\tau x_0)$ on $\Ga$ is Lipschitz equivalent to $d_S$.
\end{proposition}

\begin{proposition}
Let $G$ be a locally compact group which is second countable and compactly generated.
Fix a compact subgroup $K\subset G$.
Fix Haar measures on $G$ and $K$, such that $\vol(K)=1$.

Then the topology on $G$ is generated by a left $G$-invariant and right $K$-invariant path metric $d_G$.
The topology on the space $X=G/K$ is generated by a $G$-invariant path metric $d_X$.
\end{proposition}

\begin{proof}The proof os a variation of (3.3) in \cite{Lub}.
By the Theorem of Struble \cite{Struble}, the topology on $G$ is generated by a left-invariant proper metric $d$.
The metric
$$
d^K(x,y)=\int_Kd(xk,yk)\,dk
$$
has the same properties and is right $K$-invariant.
Let $\Om\subset G$ be a compact symmetric generating set of $G$.
Define
$$
d_G(x,y)=\inf\left\{\sum_{j=0}^{n-1}d^K(x_{j},x_{j+1}):n\in\N, x=x_0, x_n=y, x_j\in G, x_{j+1}^{-1}x_j\in\Om\right\}
$$
has the desired properties, where the properness is verified as in \cite{Lub}.
The same construction with $d^K$ replaced by 
$$
d'(xK,yK)=\int_K\int_K d(xk,yl)\,dk\,dl
$$
yields the wanted metric $d_X$ on the space $X$.
\end{proof} 

\begin{definition}
For $g\in G$ 
let 
$$
\ell(g)=\inf_{x\in X}d_X(gx,x).
$$
\end{definition}

\begin{lemma}\label{lem1.3}
Let $G$ be as in Proposition \ref{prop1.3} and let $\Ga\subset G$ be a cocompact discrete subgroup with a free point in $X=G/K$ and let $\chi:\Ga\to\GL(V_\chi)$ be a group homomorphism, where $V_\chi$ is a finite-dimensional complex vector space.
Then there exist $A,B>0$ such that for every $\ga\in\Ga$ we have
$$
|\tr\chi(\ga)|\le A e^{B\ell(\ga)}\quad\text{and}\quad \norm{\chi(\ga)}_\op\le e^{Bd_G(\ga,e)},
$$
where $\norm._\op$ denotes the operator norm.
\end{lemma}

\begin{proof}
As $\Ga\bs G$ is compact and $\Ga$ is discrete, the group $\Ga$ is finitely generated.
Let $S$ be a finite symmetric set of generators and let $\ell_w$ denote the corresponding word length function on $\Ga$.
For a given minimal representation $\ga=s_1\cdots s_n$ of $\ga\in\Ga$ we have
$$
\norm{\chi(\ga)}_\op\le\norm{\chi(s_1)}_\op\cdots\norm{\chi(s_n)}_\op\le M^n=M^{\ell_w},
$$
where $M=\max_{s\in S}\norm {\chi(s)}_\op$.
By Proposition \ref{prop1.3}, this gives the estimate on $\norm{\chi(\ga)}_\op$.

For the first estimate, let $\ga\in\Ga$ and let $x_0=g_0K\in X$ with $\ell(\ga)\ge d_X(\ga x_0,x_0)-1$.
Then
\begin{align*}
\ell(\ga)&\ge \int_K\int_K d_G(\ga g_0k,g_0l)\,dk\,dl-1\\
&\ge \int_K\int_Kd_G(\ga g_0k,g_0)-d_G(g_0,g_0l)\,dk\,dl-1\\
&\ge\int_K\int_Kd_G(\ga g_0,g_0)-\underbrace{d_G(\ga g_0,\ga g_0k)}_{d_G(e,l)}\,dk\,dl-\int_K \underbrace{d_G(g_0,g_0k)}_{d_G(e,k)}\,dk-1\\
&\ge d_G(\ga g_0,g_0)-C
\end{align*}
For some constant $C$.
Fix a compact fundamental domain $F\subset G$ for $\Ga$ and $\tau$ such that $g_0=\tau y_0$ for some $y_0\in F$.
Then
\begin{align*}
d_G(\ga g_0,g_0)
&=d_G(\tau^{-1}\ga \tau y_0,y_0)\\
&\ge d_G(\tau^{-1}\ga \tau y_0,e)-d_G(e,y_0)\\
&\ge d_G(\tau^{-1}\ga \tau ,e)-\underbrace{d_G(\tau^{-1}\ga \tau ,\tau^{-1}\ga \tau y_0)}_{d_G(e,y_0)}-d_G(e,y_0)\\
&\ge d_G(\tau^{-1}\ga \tau ,e)-D
\end{align*}
for some constant $D>0$, which depends on $F$ only.
Therefore we have
\begin{align*}
|\tr\chi(\ga)|&=|\tr(\chi(\tau^{-1}\ga\tau)|\\
&\le\dim(\chi)\,\norm{\tau^{-1}\ga\tau}_\op\\
&\le \dim(\chi)\,e^{Bd_G(\tau^{-1}\ga\tau,e)}\\
&\le \dim(\chi)\,e^{B\ell(\ga)+C+D}\mqed
\end{align*}
\end{proof}

\begin{definition}
Let $E=E_\chi=\Ga\bs(G\times V_\chi)$, where $\Ga$ acts diagonally.
The projection onto the first factor makes $E$ a complex vector bundle over $\Ga\bs G$.
Choose a hermitian metric on $E$ to define the space $L^2(E)$ of $L^2$-sections.
This space can be identified with the space $L^2(\Ga\bs G,\chi)$ of all measurable functions $f:G\to V_\chi$ with $f(\ga x)=\chi(\ga)f(x)$ and $\int_{\Ga\bs G}\sp{f(x),f(x)}_x\,dx<\infty$ (modulo nullfunctions).
The group $G$ acts by right translations on the Hilbert space $L^2(\Ga\bs G,\chi)$.
This representation is continuous but in general not unitary.
\end{definition}

\begin{definition}
Let $\what G_\adm$ denote the admissible dual of $G$, which can be viewed as the set of isomorphy classes of irreducible Harish-Chandra modules.
By Casselman's Subrepresentation Theorem each $\pi\in\what G_\adm$ can be globalised to a Hilbert representation of $G$.
We shall occasionally identify $\pi$ with this globalisation.
\end{definition}

\begin{proposition}[The trace formula]
There exists a complete $G$-stable filtration of $L^2(\Ga\bs G,\chi)$ with irreducible quotients, such that each $\pi\in\what G_\adm$ occurs with finite multiplicity $N_{\Ga,\chi}(\pi)\in\N_0$ in the induced graded representation and such that for each $f\in C_c^\infty(G)$ the trace formula
$$
\sum_{\pi\in\what G_\adm}N_{\Ga,\chi}(\pi)\ \tr\pi(f)=\sum_{[\ga]}\vol(\Ga_\ga\bs G_\ga)\ \CO_\ga(f)\ \tr\chi(\ga)
$$
holds, where $G_\ga$ and $\Ga_\ga$ are the centralisers of $\ga$ in $G$ and $\Ga$ respectively. 
Further, $\CO_\ga(f)$ is the \e{orbital integral}
$$
\CO_\ga(f)=\int_{G/G_\ga}f(x\ga x^{-1})\,dx.
$$
\end{proposition}

\begin{proof}
This is Theorem 3 in \cite{NonUnitaryTrace}.
\end{proof}

\section{Lie groups}
\begin{definition}
Let $G$ be a connected Lie group and $K$ a compact subgroup.
Let $\g_\R$ denote its real Lie algebra and 
let $\g=\g_\R\otimes_\R\C$ be its complex Lie algebra.
The universal envelope $U(\g)$ will be identified with the algebra of all left-invariant differential operators on $G$.
For $k\in\N$, let $U_k(\g)$ denote the finite-dimensional space of all $D\in U(\g)$ of degree $\le k$.
\end{definition}

\begin{definition}
Fix a $G$-left-invariant and $K$-right-invariant Riemannian metric on $G$ and let $d_G$ denote the induced distance function on $G$.
For $r>0$ let $V(r)=\vol(B_r(e))$ be the volume of the open ball of radius $r$ around $e$.
As the sectional curvature of $G$ is bounded below, there are $A,\be>0$ such that
$$
V(r)\le Ae^{\be r}.
$$
For this, see \cite{BC}.
\end{definition}

\begin{lemma}\label{lem2.3}
For every $\eps>0$ one has $\sum_{\ga\in\Ga}e^{-(\be+\eps)d(\ga,e)}<\infty$.
\end{lemma}

\begin{proof}
Let $U$ be a compact unit neighbourhood such that $U\cap\ga U=\emptyset$ for every $\ga\in\Ga\sm\{e\}$.
Then there exists $c_1>0$ with
\begin{align*}
\sum_{\ga\in \Ga}e^{-(\be+\eps)d(\ga,e)}
&\le c_1\sum_{\ga\in\Ga}\int_Ue^{-(\be+\eps)d(\ga u,e)}\,du\\
&\le c_1\int_Ge^{-(\be+\eps)d(x,e)}\,dx<\infty.
\qedhere
\end{align*}
\end{proof}

\begin{definition}\label{def2.4}
Choose $B>0$  such that the conlcusions of Lemma \ref{lem1.3} are satisfied.
Fix some $C>2(\be+B)$.
For $k\in\N$  let $C^{k,C}(G)$ denote the space
of all $k$-times continuously differentiable functions $f:G\to\C$ such that
$$
p_D(f)=\sup_{x\in G}e^{Cd_G(x,e)}|Df(x)|<\infty
$$
holds for every $D\in U_k(\g)$.
Equip $C^{k,C}(G)$ with the topology given by the seminorms $p_D$, $D\in U_k(\g)$.
Note that, as $C>\be$, for $f\in C^{k,C}(G)$ we have $Df\in L^1(G)$ for every $D\in U_k(\g)$.
\end{definition}

\begin{proposition}\label{prop2.4}
Let $N\in\N$ with $2N>\dim G$.
Then the trace formula is valid for every  $f\in C^{2N,C}(G)$.
Let $\tr_\Ga(f)$ denote either side of the trace formula. Then $\tr_\Ga$ is a continuous linear map $C^{2N,C}(G)\to\C$.
\end{proposition}

{\it Proof.} We need a lemma.

\begin{lemma}
Let $R_\chi$ denote the right translation representation of $G$ on $L^2(\Ga\bs G,\chi)$.
Equip the algebra of bounded operators $\CB\(L^2(\Ga\bs G,\chi)\)$ with the strong operator topology. Then for $f\in C^{0,C}(G)$ and $\phi\in L^2(\Ga\bs G,\chi)$, the  integral
$$
R_\chi(f)\phi(x)=\int_Gf(y)\phi(xy)\,dy
$$
exists almost everywhere for $x\in G$ and defines a Hilbert-Schmidt operator $R_\chi(f)\in\CB\(L^2(\Ga\bs G,\chi)\)$.
\end{lemma}

\begin{proof}
Let $F\subset G$ denote a compact fundamental domain for $\Ga$, i.e., $F$ is compact and contains a measurable set $V$ of representatives for $\Ga\bs G$ such that $F\sm V$ has measure zero.
Assume further that $e\in F$.
For $x,y\in F$ and $\ga\in\Ga$ one has
$
d_G(\ga,e)\le d_G(\ga,\ga x)+d_G(\ga x,y)+d_G(y,e)$ and
hence
\begin{align*}
d_G(x^{-1}\ga y,e)&=d_G(\ga y,x)\\
&\ge d_G(\ga,e)-d_G(e,x)-d_G(y,e)\\
&\ge d_G(\ga,e)-2\diam(F).
\end{align*}
It follows
\begin{align*}
|f(x^{-1}\ga y)|\norm{\chi(\ga)}_\op
&\le |f(x^{-1}\ga y)|\ e^{Bd_G(\ga,e)}\\
&\le T_1e^{-Cd_G(x^{-1}\ga y,e)}\ e^{Bd_G(\ga,e)}\\
&\le Te^{-2\be d(\ga,e)}
\end{align*}
for some $T_1,T>0$.
By Lemma \ref{lem2.3}, the sum $k(x,y)=\sum_{\ga\in\Ga}f(x^{-1}\ga y)\chi(\ga)$
converges uniformly absolutely, therefore represents a continuous integral kernel $k$ on the compact manifold $\Ga\bs G$ for the operator $R_\chi(f)$. Having a continuous, hence square-integrable kernel, this operator is Hilbert-Schmidt.
\end{proof}

\begin{proof}[Proof of Proposition \ref{prop2.4}]
Let $\Delta$ denote the left-invariant group Laplacian given by the Riemannian metric on $G$.
 Since $\Delta\in U(\g)$, the operator $R_\chi(\Delta)$ is well defined as a densely defined unbounded operator on $L^2(\Ga\bs G,\chi)$. 
 Then, for $2N>\dim(G)$ and $c\in\C$, which is not in the spectrum on $\Delta$, the operator 
$R(\Delta+c)^{-N}$ is of trace class. 
This is folklore and can be deduced from \cite{Sh}.
One way to see it directly is to choose a smooth atlas $(U_i,\phi_i)_{i\in I}$ such that the bundle $E_\chi$ trivialises over each $U_i$, where the charts $\phi_i$ are maps taking values in $\R^n/\Z^n$. Then one uses a smooth partition of unity underlying $(U_i)_{i\in I}$ to translate the problem to $\R^n/\Z^n$ where it can be solved  using Fourier-series.

Let $\om :[0,\infty )\to[0,1]$ be a monotonic $C^{2N}$-function
with compact support, $\om\equiv 1$ on $[0,1]$ and
$|\om^{(k)}(t)|\le 1$ for $k=1,\dots ,2N$. Let
$h_n(x):=\om\(\frac{d_G(x,e)}{n}\)$ for $x\in G$, $n\in \N$. Then
$|Dh_n(x)|\le \frac{C_D}{n}$ for any $D\in\g U(\g)$. Let
$f_n=h_nf$ then $f_n\to f$ locally uniformly.
As $f_n$ has compact support, we  get that $R_\chi(f_n)$ is a trace class operator.

Denoting the trace norm by $\norm{.}_{tr}$ we infer
\begin{align*}
\norm{R(f_n)-R(f)}_{tr}
&= \norm{R(\Delta +1)^{-N}(R((\Delta +1)^Nf_n)-R((\Delta +1)^Nf))}_{tr}\\
&
\le \norm{R(\Delta +1)^{-N}}_{tr}\norm{R((\Delta +1)^Nf_n)-R((\Delta
+1)^Nf)}_\op.
\end{align*}
{\it Claim.} For the operator norm  we have
$$
\parallel R(Df_n) -R(Df) \parallel_\op \to 0
$$
as $n\to \infty$ for any $D\in U(\g)$ of degree $\le 2N$.

{\it Proof of the claim.} By the Poincar\'e-Birkhoff-Witt theorem
$D(f_n)=D(h_nf)$ is a sum of expressions of the type
$D_1(h_n)D_2(f)$ and $D_1$ can be chosen to be the identity
operator or in $\g U(\g)$. The first case gives the summand
$h_nD(f)$ and it is clear that $\norm{R(Df)-R(h_nDf)}_\op$ tends to zero
as $n$ tends to infinity. For the rest assume $D_1\in \g U(\g)$.
Then $|D_1(h_n)(x)|\le \frac{c}{n}$ hence
$\norm{R(D_1(h_n)D_2(f))}_\op$ tends to zero because $D_2(f)$ is in
$C^{0,C}(G)$. The claim follows.

We conclude that $\tr R(f_n)$ tends to $\tr R(f)$
as $n\to \infty$, which implies that the spectral side of the trace formula defines a continuous functional on $C^{2N,C}(G)$. 
It remains to show that it coincides with the geometric side.
For this we go back to the situation where $f_n=h_nf$.
It follows
\begin{align*}
\sum_{\pi\in\what{G}_\adm} N_{\Ga ,\omega}(\pi)\tr\pi (f) &= \tr R(f)= \lim_n \tr R(f_n).\\
\end{align*}

Now suppose as an additional condition that $f\ge 0$ and $\chi =1$.
Then we are allowed to interchange the limit and the sum by monotone convergence and thus in this case
$$
\sum_{\pi\in\what{G}_\adm} N_{\Ga ,1}(\pi)\tr\pi (f)=
\sum_{[\ga]} \vol(\Ga_\ga\bs G_\ga)\ \CO_\ga(f).
$$
In particular, the right hand side is finite.
For general $\chi$, we use the bound on $\tr\chi(\ga)$ to justify the same interchange by dominated convergence using the result for $\chi =1$.
As both sides of the formula are linear in $f$, the claim follows generally.
\end{proof}

\section{The Lefschetz formula}\label{Sec3}
In this section, we adapt the contents of Section 1 in \cite{GAFA} to the situation of non-unitary twists.
For the most part, it consists of definitions.
At the end, we only give the changes necessary to adapt the proof of the main theorem of Section 1 in [loc.cit.] to our situation.

Let $G$ be a semi-simple, connected Lie group with finite center.
Let $H=AT$ be a maximally split Cartan subgroup, where $A$ is the  split component and $T$ is compact abelian.
Let $\a$ be real Lie algebra of $A$ and $\a^*=\Hom(\a,\R)$ be its dual space.
Fix a parabolic subgroup $P$ of $G$ with Langlands decomposition $P=MAN$.
The choice of $P$ induces a choice of positive roots $\Phi^+(\a,\g)\subset \a^*$.
Let $\a^-=\big\{X\in \a: \al(a)<0\ \forall_{\al>0}\big\}$ and $A^-=\exp(\a^-)$ be the negative cones in $\a$ and $A$.

\begin{definition}
For $\mu\in\a^*$ and $j\in\N$, let $C^{\mu,j}(A^-)$
denote the space of $j$-times continuously differentiable functions $\phi$ on $A$ which vanish outside $A^-$ and satisfy
$$
N_D(\phi)=\sup_{a\in A}\,|D\phi(a)|a^{-\mu}<\infty
$$
for every $D\in U_j(A)$.
Equip $C^{\mu,j}(A^-)$ with the topology induced by the seminorms $N_D$, $D\in U_j(A)$.
Let $D_1,\dots,D_r$ be a basis of $U_j(A)$, then the norm
$$
\norm\phi=N_{D_1}(\phi)+\dots+N_{D_r}(\phi)
$$
induces the topology, hence $C^{\mu,j}(A^-)$ is a Banach space.
\end{definition}

\begin{definition}
Choose a maximal compact subgroup $K$ with
Cartan involution $\theta$, i.e., $K$ is the group of
fixed points of $\theta$. 
After conjugation we can assume that $A$ and $M$ are stable under
$\theta$. Then $A$ is a maximal split torus of $G$ and $M$ is a subgroup of $K$. The centralizer of $A$
is $AM$. Let $W(A,G)$ be the \emph{Weyl group} of $A$, i.e. $W$ is the quotient of the
normalizer of $A$ by the centralizer. This is a finite group acting on $A$. 
\end{definition}

\begin{definition}\label{def3.3}
We have to fix Haar measures. We use the normalization of
Harish-Chandra \cite{HC-HA1}. Note that this normalization
depends on the choice of an invariant bilinear form $B$ on
$\g_\R$ which we keep at our disposal until later. Changing $B$ amounts to scaling the
metric of the symmetric space. Note further that in this normalization of Haar measures the
compact groups $K$ and $M$ have volume $1$.
\end{definition}

\begin{definition}
We write $\g_\R, \mathfrak{k}_\R,\a_\R,\m_\R,\n_\R$ for the real Lie
algebras of the groups $G,K,A,M,N$ and $\g,\mathfrak{k},\a,\m,\n$ for their
complexifications. $U(\g)$ is the universal enveloping
algebra of $\g$. This algebra is isomorphic to the algebra
of all left invariant differential operators on $G$ with
complex coefficients. Pick a Cartan subalgebra $\mathfrak{t}$ of $\m$. Then $\h=\a\oplus\mathfrak{t}$ is a
Cartan subalgebra of $\g$. Let $W(\h,\g)$ be the corresponding absolute Weyl group.

Let $\a^*$ denote the dual space of the complex vector
space $\a$. Let $\a_\R^*$ be the real dual of $\a_\R$. We identify $\a_\R^*$ with the real vector space of all $\la\in\a^*$ that map $\a_\R$ to $\R$. Let $\Phi\subset\a^*$ be the set of all roots
of the pair $(\a,\g)$ and let $\Phi^+$ be the subset of
positive roots with respect to $P$. Let $\Delta\subset
\Phi^+$ be the set of simple roots. Then $\Delta$ is a basis of $\a^*$. The open {\it
negative Weyl chamber} $\a_\R^-\subset\a_\R$ is the cone of all
$X\in\a_\R$ with $\alpha(X)<0$ for every $\alpha\in\Delta$. Let $\overline{\a_\R^-}$ be the
closure of $\a_\R^-$. The $W(A,G)$-translates $w\a_\R^-$ of $\a_\R^-$ are pairwise disjoint
and their union equals $\a_\R$ minus a finite number of hyperplanes. This is called the
\emph{regular set},
$$
\a_\R^{reg}= \bigcup_{w\in W(A,G)} w\a_\R^-.
$$
Let $A^{reg}=\exp\left(\a_\R^{reg}\right)$ be the regular set in $A$. The elements are
called the regular elements of $A$. A given $a\in A$ lies in $A^{reg}$ if and only if the
centralizer of $a$ in $G$ equals the centralizer of $A$ in $G$ which is $AM$.
\end{definition}

The bilinear form $B$ is indefinite on $\g_\R$, but the form
$$
\sp{X,Y}= -B(X,\theta(Y))
$$
is positive definite, ie an inner product on $\g_\R$. We extend it to an inner product on
the complexification $\g$. Let $\norm X=\sqrt{\sp{X,X}}$ be the corresponding norm.
The form $B$, being nondegenerate, identifies $\g$ to its dual space $\g^*$. In this way we
also define an inner product $\sp{.,.}$ and the corresponding norm on $\g^*$.
Furthermore, if $V\subset\g$ is any subspace on which $B$ is nondegenerate, then $B$ gives
an identification of $V^*$ with $V$ and so one gets an inner product and a norm on
$V^*$. This in particular applies to $V=\h$, a Cartan subalgebra of $\g$, which is defines
over
$\R$.

\begin{definition}
Let $\Ga\subset G$ be a discrete, cocompact, torsion-free subgroup. 
We are interested in the closed geodesics on the locally symmetric space $X_\Ga=\Ga\bs
X=\Ga\bs G/K$. Every such geodesic 
$c$ lifts to a $\Ga$-orbit of geodesics on $X$ and gives a $\Ga$-conjugacy class $[\ga_c]$ of
elements closing the particular geodesics. This induces a bijection between the set of all
homotopy classes of closed geodesics in $X_\Ga$ and the set of all non-trivial conjugacy
classes in $\Ga$.
\end{definition}

\begin{definition}
An element $am$ of $AM$ is called {\it split regular} if
$a$ is regular in $A$. An element $\ga$ of $\Ga$
is called split regular if $\ga$ is in $G$ conjugate to a
split regular element $a_\ga m_\ga$ of $AM$. In that case
we may (and do) assume that $a_\ga$ lies in the negative
Weyl chamber $A^-=\exp(\a_\R^-)$ in $A$. Let $\CE(\Ga)$
denote the set of all conjugacy classes in $\Ga$ that
consist of split regular elements. Via the above correspondence the set $\CE(\Ga)$ can be
identified with the set of all homotopy classes of regular closed geodesics in $X_\Ga$, \cite{jost}.
\end{definition}

\begin{definition}\label{def3.7}
Let $[\ga]\in\CE(\Ga)$.
There is a closed geodesic $c$ in the Riemannian manifold $\Ga\bs G/K$ which gets closed by
$\ga$. This means that there is a lift $\tilde c$ to the universal covering $G/K$ which is
preserved by $\ga$ and $\ga$ acts on $\tilde c$ by a translation.
The closed geodesic $c$ is not unique in general but lies in a unique maximal flat
submanifold $F_\ga$ of $\Ga\bs G/K$. Let $\la_\ga$ be the volume of that flat,
$$
\la_\ga= \vol(F_\ga).
$$ 
Let $\Ga_\ga$ and $G_\ga$ denote the centralizers of $\ga$ in $\Ga$ and $G$ respectively.
The flat $F_\ga$ is the image of $\Ga_\ga\bs G_\ga /K_\ga$ in $\Ga\bs G/K$, where we assume
that $K$ is chosen so that $K_\ga=G_\ga\cap K$ is a maximal compact subgroup of $G_\ga$. 
In Harish-Chandra's normalization $G_\ga$ is equipped with the Haar measure that satisfies
$\int_{G_\ga}=\int_{G_\ga/K_\ga}\int_{K_\ga}$, where $K_\ga$ has the Haar measure with
$\vol(K_\ga)=1$ and $G_\ga/K_\ga$ gets the measure induced by the metric of $G/K$. Therefore
$$
\la_\ga= \vol(\Ga_\ga\bs G_\ga /K_\ga)=\vol(\Ga_\ga\bs G_\ga).
$$
Note that for not maximally split elements $\ga$ the relation between $\la_\ga$ and
$\vol(\Ga_\ga\bs G_\ga)$ is more complicated \cites{geom,HRLef}.
\end{definition}

\begin{definition}
Let $\n$ denote the complexified Lie algebra of $N$. For
any $\n$-module $V$ let $H_q(\n,V)$ and $H^q(\n,V)$ for $q=0,\dots,\dim\n$
be the Lie algebra homology and cohomology \cite{BorWall}.    
 If
$\pi\in\what G_\adm$, then $H_q(\n,\pi)$ and $H^q(\n,\pi)$ are finite
dimensional $AM$-modules \cite{HeSch}. Note that they a priori only are $(\a\oplus\m,M)$-modules, but since $A$ is isomorphic to its Lie algebra they are  $AM$-modules.
\end{definition}

Note that $AM$ acts on the Lie algebra $\n$ of $N$ by the
adjoint representation. Let $[\ga]\in\CE(\Ga)$. Since
$a_\ga\in A^-$ it follows that every eigenvalue of $a_\ga
m_\ga$ on $\n$ is of absolute value $<1$. Therefore
$\det(1-a_\ga m_\ga | \n)\ne 0$.

Fix a finite dimensional irreducible representation
$\sigma$ of $M$ and denote by $\breve\sigma$ the dual representation. A {\it quasi-character} of $A$ is a
continuous group homomorphism to $\C^\times$. Via
differentiation the set of quasi-characters can be
identified with the dual space $\a^*$. For $\la\in\a^*$ we write $a\mapsto a^\la$ for the
corresponding quasicharacter on $A$. We denote by
$\rho\in\a^*$ the modular shift with respect to $P$, i.e., for $a\in A$ we have
$\det(a|\n)=a^{2\rho}$.

\begin{definition}
For a complex vector space $V$ on which $A$ acts linearly and $\la\in\a^*$, let $(V)_\la$
denote the generalized $(\la+\rho)$-eigenspace, i.e.,
$$
(V)_\la= \{ v\in V\mid (a-a^{\la+\rho}Id)^n v=0\ \ {\rm for\ some\ } n\in\N\}.
$$
Since $H^p(\n,\pi)$ is finite dimensional, the Jordan Normal Form Theorem implies that
$$
H^p(\n,\pi) = \bigoplus_{\nu\in\a^*} H^p(\n,\pi)_\nu.
$$
Let $T$ be a Cartan subgroup of $M$ and let $\mathfrak{t}$ be its complex Lie algebra. Then $AT$ is a
Cartan subgroup of $G$. Let $\Lambda_\pi\in(\a\oplus\mathfrak{t})^*$ be a representative of the
infinitesimal character of $\pi$. By Corollary 3.32 of \cite{HeSch} it follows,
$$
H_p(\n,\pi)=\bigoplus_{\nu=w\Lambda_{\pi}|_\a}H_p(\n,\pi)_\nu,
$$
where $w$ ranges over $W(\g,\h)$.
\end{definition}

\begin{lemma}\label{1.1}
For $0\le p\le d=\dim(\n)$ we have
$$
 H_p(\n,\pi)\ \cong\ H^{d-p}(\n,\pi)\otimes\det(\n),
$$
where the determinant of a
finite dimensional space is the top exterior power. So $\det(\n)$ is a one dimensional
$AM$-module on which $AM$ acts via the quasi-character $am\mapsto \det(am|\n)=a^{2\rho}$.
This in particular implies
$$
H^p(\n,\pi)=\bigoplus_{\nu= w\Lambda_\pi |_\a}H^p(\n,\pi)_{\nu-2\rho}.
$$
\end{lemma}

\begin{proof}
The first part follows straight from the definition of Lie algebra cohomology. The
second part by Corollary 3.32 of \cite{HeSch}.
\end{proof}

\begin{definition}
For $\la\in\a^*$
 and $\pi\in\what G_\adm$ let
$$
m_\la^\sigma(\pi)= \sum_{q=0}^{\dim \n} (-1)^{q+\dim\n} \dim\left(
H^q(\n,\pi)_\la\otimes\breve\sigma\right)^M,
$$
where the superscript $M$ indicates the subspace of $M$-invariants. Then 
$m_\la^\sigma(\pi)$ is an integer and by the above, the set of $\la$ for which
$m_\la^\sigma(\pi)\ne 0$ for a given $\pi$ has at most $|W(\g,\h)|$ many elements. 
\end{definition}

\begin{definition}
For  $\mu\in\a^{*}$ and $j\in\N$ let $\CC^{j,\mu,-}(A)$ denote the space of functions
$\phi$ on $A$ which \nopagebreak
\begin{itemize}
  \item are $j$-times continuously differentiable on $A$,
  \item are zero outside $A^-$,
  \item are such that $a^{-\mu} D\phi(a)$ is bounded on $A$  for every invariant differential operator $D$ on $A$ of
degree $\le j$.
\end{itemize}

For every invariant differential operator $D$ of degree $\le j$ let $N_D(\phi)=\sup_{a\in A}\left| a^{-\mu}\, D\phi(a)\right|$.
Then $N_D$ is a seminorm.
Let $D_1,\dots,D_n$ be a basis of the space of invariant differential operators of degree $\le j$, then
$N(\phi)=\sum_{j=1}^n N_{D_j}(\phi)$ is a norm that makes $\CC^{j,\mu,-}(A)$ into a Banach space.
A different choice of basis will give an equivalent norm.
\end{definition}

\begin{definition}
For $[\ga]\in\CE(\Ga)$ let
$$
\ind(\ga)=\frac{\la_\ga }{\det(1-a_\ga m_\ga \mid \n)}\ >\ 0.
$$
\end{definition}

\begin{theorem}\label{lefschetz} (Lefschetz Formula)\\
There exists $j\in\N$ and $\mu\in\a^{*}$ such that
for any $\phi\in \CC^{j,\mu,-}(A)$ we have that
$$
\sum_{\pi\in\what G_\adm} N_{\Ga,\chi}(\pi)\sum_{\la\in\a^*} m_\la^\sigma(\pi) \int_{A^-}\phi(a)
a^{\la+\rho} da
$$
equals 
$$
\sum_{[\ga]\in\CE(\Ga)}\ind(\ga)\,
\tr\sigma(m_\ga)\,\tr(\chi(\ga))\,\phi(a_\ga),
$$
where all sums and integrals converge absolutely. The inner sum on the left is always finite,
more precisely it has length $\le |W(\h,\g)|$. The left hand side is called the
\emph{global side}, as opposed to the \emph{local side} on the right. 
The map that sends $\phi$ to either side of the formula is continuous linear functional on the Banach space $\CC^{j,\mu,-}(A)$.
\end{theorem}

\begin{proof}
The proof is essentially the same as that of Theorem 1.2 in \cite{GAFA}, except for a missing estimate which justifies the factor $e^{C d_G(x,e)}$, needed in the non-unitary case as in Definition \ref{def2.4} and Proposition \ref{prop2.4}.
This is done as follows: By the definition of the test function $f$, one only needs to consider $x=kn\,am\,(kn)^{-1}$ for $am\in AM$, $k\in K$ and $n\in N$, where the latter stays in a fixed compact subset of $N$.
Using left $G$-and right$K$-invariance, one gets
\begin{align*}
d_G(kn\,am\,(kn)^{-1},e)&=d_G(n\,am\,n^{-1},e)\\
&=d_G(a,n^{am}n^{-1}m^{-1}),
\end{align*}
where the element $\om=kn\,am\,(kn)^{-1}$ stays in a fixed compact set $\Om\subset G$.
Then
$$
d_G(a,\om)\le d_G(a,e)+d_G(e,\om)\le d_G(a,e)+c
$$
for a constant $\al>0$.
Taking $e^{Cd_G(x,e)}$, the constant $\al$ contributes a multiplicative constant, which plays no role.
Next let $a=\exp(X)$ and $\ga(t)=\exp(tX)$.
then $d_G(a,e)=\int_0^1\norm{\ga'(t)}\,dt=\int_0^1\norm X\,dt=\norm X$.
Now finally $a^{-\mu}=\norm X^{-\mu}$, which shows that one achieves the boundedness of $e^{Cd_g(x,e)}|Df(x)|$ by increasing $\mu$.
\end{proof}

\section{Zeta functions of Ruelle and Selberg}\label{sec2}
From now on we assume that the splitrank of $G$ be one, i.e., $\dim A=1$.

In this case, $\CE(\Ga)$ coincides with the set of all non-trivial conjugacy classes in $\Ga$.

\begin{definition}
A class $[\ga]\in\CE(\Ga)$ is called \e{primitive}, if for every $\tau\in\Ga$ and every $n\in\N$ one has
$$
\ga=\tau^n\quad\Rightarrow\quad n=1.
$$
For every $\ga\in\Ga\sm\{1\}$ there exists a uniquely determined primitive $\ga_0\in\Ga$ such that 
$$
\ga=\ga_0^m
$$
for some $m\in\N$.
we call $\ga_0$ the \e{underlying primitive} of $\ga$.
\end{definition}

\begin{definition}
For $a\in A$  we write $\ell(a)=|2\rho(\log a)|$.
Note that for $\ga\in\Ga\sm\{1\}$, the number $\ell(a_\ga)$ is the length of the closed geodesic in the locally symmetric space $\Ga\bs G/K$, whose free homotopy class corresponds to the $\Ga$-conjugacy class of $\ga$. 
We also write
$$
\ell(\ga)
$$
for $\ell(a_\ga)$.
Note that
$$
\la_\ga=\ell(\ga_0),
$$
where $\ga_0$ is the underlying primitive of $\ga$ and $\la_\ga$ is as in Definition \ref{def3.7}.

\end{definition}

\begin{definition}
For $s\in\C$ with $\Re(s)>> 0$ the \e{Ruelle zeta function} is given by the  product
$$
R_\chi(s)=\prod_{\substack{[\ga]\\ \text{primitive}}}\(1- \chi(\ga)\,e^{-s\ell(\ga)}\)
$$
We shall show, that $R_\chi(s)$ converges locally uniformly in some half plane of the form $\{\Re(s)>A\}$ and that it admits an analytic continuation to a meromorphic function on $\C$.
The later will come about through its relation to the \e{Selberg zeta function}, which for a finite-dimensional representation $\sigma$ of the compact group $M$ is defined by
$$
Z_{\sigma,\chi}(s)=\prod_{\substack{[\ga]\\ \text{primitive}}}\prod_{k=0}^\infty\det\(1-\sigma(m_\ga)\otimes \chi(\ga)\otimes S^k\[\Ad(m_\ga a_\ga)|_{\n}\]\,e^{-s\ell(\ga)}\).
$$
\end{definition}

\begin{proposition}
Assuming convergence for the moment, we have the identity
$$
R_\chi(s)=\prod_{k=0}^{\dim\n}Z_{\bigwedge^k_\n,\chi}(s-j)^{(-1)^k}.
$$
Here $\bigwedge _\n^k$ stands for the representation of $M$ on the $k$-exterior power $\bigwedge^k\n$ of the space $\n$.
Further, the product defining $R_\chi$ converges, if the products of the $Z_{\bigwedge_\n^k,\chi}(s)$ converge absolutely for every $k=0,\dots,\dim\n$.
\end{proposition}

\begin{proof}
A straightforward computation shows that in the range of absolute convergence we have
\begin{align*}
\frac{Z_{\sigma,\chi}'}{Z_{\sigma,\chi}}(s)
&=\sum_{[\ga]}\ind(\ga)\ \tr(\chi(\ga))\ \tr\sigma(m_\ga)\ e^{-s\ell(\ga)}\\
&=\sum_{[\ga]}\frac{\ell(\ga_0)}{\det(1-a_\ga m_\ga|\n)}\ \tr(\chi(\ga))\ \tr\sigma(m_\ga)\ e^{-s\ell(\ga)}
\end{align*}
and
$$
\frac{R_\chi'}{R_\chi}(s)=\sum_{[\ga]}\ell(\ga_0)\ \tr\chi(\ga)\ e^{-s\ell(\ga)}.
$$
For any linear map $T$ on a finite-dimensional vector space one has $\det(1-T)=\sum_{k=0}^{\dim V}(-1)^k\tr\(\bigwedge^kT\)$.
Applying this to the action of $a_\ga m_\ga$ on $\n$ one gets
$$
\frac{R_\chi'}{R_\chi}(s)
=\sum_{k=0}^{\dim\n}(-1)^k\frac{Z_{\bigwedge^k_\n,\chi}'}
{Z_{\bigwedge^k_\n,\chi}}(s).
$$
Hence the logarithmic derivatives of the two sides of the claimed equation coincide.
Hence the two meromorphic functions coincide up to a multiplicative constant, which equals 1, since both functions tend to 1, as $s\to +\infty$.
\end{proof}

\begin{theorem}
The product of the Selberg zeta function converges in some half-plane and the so defined holomorphic function extends to a meromorphic function on $\C$.
The (vanishing) order at $s\in\C$ is given by the finite sum of integers
$$
\sum_{\pi\in\what G_\adm}N_{\Ga,\chi}(\pi)\ m_{(s-1)\rho}(\pi).
$$
\end{theorem}

\begin{proof}
We need to normalize Haar measures.
These depend on the choice of the form $B$ as in Definition \ref{def3.3}.
We normalise  $B$ such that the set
$$
\big\{ a\in A \mid 0\le 2\rho(\log a)\le 1\big\}
$$
has volume $1$.

For $s\in\C$ and $j\in\N$ define
$$
L^j(s)=\sum_{[\ga]\in\CE(\Ga)} \ind(\ga)\ \tr(\chi(\ga))\ \tr\sigma(m_\ga)\ \ell(\ga)^{j+1}\ e^{-s\ell(\ga)}.
$$
Then we have
$$
L^j(s)=D^{j+1}\frac{Z_{\sigma,\chi}'}{Z_{\sigma,\chi}}(s),
$$
where $D$ denotes the differential operator
$
D= -\frac \partial{\partial s}.
$
We shall show that this series converges absolutely for
$\Re(s)>A$ for some $A>0$. 
As it is a (generalised) Dirichlet series, it then converges locally uniformly.

\begin{definition} 
For given $\pi\in\what G_\adm$ let $\Lambda(\pi)$ denote the set of all $\la\in\a^*$ with
$m_{\la -\rho}(\pi)\ne 0$. Then $\Lambda(\pi)$ has at most $|W(\h,\g)|$ elements.
\end{definition}

For given
$\la\in\a^*$ there exists a uniquely determined number $[\la]\in\C$ with $\la=[\la]\,2\rho$.

\begin{lemma}
There exists $A>0$, such that for $j\in\N$ large enough, the series $L^j(s)$ converges locally uniformly in the set
$$
\big\{ s\in\C : \Re(s)>A\big\}.
$$
The function $L^j(s)$ can be written as Mittag-Leffler series,
\begin{align*}
L^j(s)&=\sum_{\pi\in\what G_\adm} N_{\Ga,\chi}(\pi)
\sum_{\la\in\Lambda(\pi)} m_{\la-\rho}(\pi) D^{j+1}\frac 1{s+[\la]}.
\end{align*}
The double series converges locally uniformly on $\C$ minus its discrete set of poles. 
These poles all lie in $\{\Re(s)\le A\}$.
\end{lemma}

\begin{proof}
We will show that the series $L^j(s)$ converges if $\Re(s_k)$ is
sufficiently large. 
For $a\in A$ set
$$
\phi(a)= \ell(a)^{j+1}\ a^{2s\rho}.
$$
For $\Re(s)>>0$, the Lefschetz formula is valid for this test function. The local
side of the Lefschetz formula for $\sigma= triv$ equals
$$
\sum_{[\ga]\in\CE(\Ga)} \ind(\ga)\,\tr\(\chi(\ga)\)\, \ell(a_\ga)^{j+1}\ a_\ga^{2s\rho}=  L^j(s).
$$
The convergence assertion in the Lefschetz formula implies that the series  converges
absolutely  if
$\Re(s)>A$ for some $A\in\R$.  
Since $L^j(s)$ is a Dirichlet series, it will
converge locally uniformly in that half-plane.
The theorem will follow, if we can show that the spectral side of the trace formula equals the Mittag-Leffler series above, since the trace formula grants convergence in the half plane and a Mittag-Leffler series, which converges at some point in $\C$, automatically converges locally uniformly outside its set of poles.

 With our given test
function and the Haar measure chosen we compute
\begin{align*}
\int_{A^-}\phi(a) a^\la da &=(-1)^{j+1} \int_{A^-} (2\rho(\log a))^{j+1} a^{s\rho+\la} da\\
&= (-1)^{j+1} \int_0^\infty  {t}^{j+1}
e^{-(s+[\la])t} d t\\
&= D^{j+1} 
\int_0^\infty  e^{-(s+[\la])t}\, d t\\
&= D^{j+1}  \frac
1{s+[\la]}.
\end{align*}
Writing $m_\la=m_\la^{triv}$ and performing a $\rho$-shift we see that the Lefschetz formula
gives
\begin{eqnarray*}
L^j(s)&=&\sum_{\pi\in\hat G}N_{\Ga,\chi}(\pi)\sum_{\la\in\a^*} m_{\la-\rho}(\pi)\, D^{j+1} \frac
1{s+[\la]}
\end{eqnarray*}
for $\Re(s)>A$.
The lemma follows.
\end{proof}

The convergence assertion in the theorem follows from the convergence of $L^j(s)$, since there are only finitely many classes $[\ga]$ with $\ell(\ga)<1$.
The rest is clear.
\end{proof}

{\small Mathematisches Institut\\
Auf der Morgenstelle 10\\
72076 T\"ubingen\\
Germany\\
\tt deitmar@uni-tuebingen.de}

\today

\end{document}